\documentclass[11pt,reqno,a4paper,fullpage]{amsart}
\usepackage{amsmath,amssymb,latexsym,amsthm}
\usepackage{scalerel}
\usepackage{enumerate}
\usepackage{url,graphics,tikz}
\usepackage{orcidlink}
\usetikzlibrary{shapes.geometric,decorations.pathmorphing}
\usepackage{color}
\makeatletter
\@namedef{subjclassname@2020}{%
 \textup{2020} Mathematics Subject Classification}
\makeatother
\newtheorem{thm}{Theorem}[section]

\newtheorem{cor}[thm]{Corollary}
\newtheorem{lem}[thm]{Lemma}
\newtheorem{prop}[thm]{Proposition}

\newtheorem{Def}[thm]{Definition}

\numberwithin{equation}{section}
\usetikzlibrary{
                shapes,
                arrows,
                automata,
                shadows} 
\tikzset{sstate/.style={state, 
                inner sep=2pt, 
                minimum size=7pt,
                circular drop 
                shadow,fill=white}}
\def\NN{\mathbb{N}}
\def\QQ{\mathbb{Q}}
\def\RR{\mathbb{R}}
\def\ZZ{\mathbb{Z}}
\def\mm{{\mathcal M}}
\def\pp{{\mathcal P}}

\def\ds{\displaystyle}
\begin{document}
%
  
\title{On Sum Graphs over Some Magmas}


\author{António Machiavelo~\orcidlink{0000-0002-7595-7275}}
\email[corresponding author]{antonio.machiavelo@fc.up.pt}

\author{Rog\'{e}rio Reis~\orcidlink{0000-0001-9668-0917}}
\email{rogerio.reis@fc.up.pt}

\address{CMUP, Faculdade de Ciências da Universidade do Porto,
  4169-007 Porto, Portugal}

\begin{abstract}
  We consider the notions of \emph{sum graph} and of \emph{relaxed sum
    graph} over a magma, give several examples and results of these
  families of graphs over some natural magmas. We classify the cycles
  that are sum graphs for the magma of the subsets of a set with the
  operation of union, determine the abelian groups that provide a sum
  labelling of $C_4$, and show that $C_{4\ell}$ is a sum graph over
  the abelian group $\ZZ_f\times\ZZ_f$, where $f=f_{2\ell}$ is the
  corresponding Fibonacci number. For integral sum graphs, we give a
  linear upper bound for the radius of matchings, improving Harary's
  labelling for this family of graphs, and give the exact radius for
  the family of totally disconnected graphs.

  We found integer labellings for the 4D-cube, giving a negative
  answer to a question of Melnikov and Pyatikin, actually showing that
  the 4D-cube has infinitely many primitive labellings. We have also
  obtained some new results on mod sum graphs and relaxed sum
  graphs. Finally, we show that the direct product operation is closed
  for strong integral sum graphs.
\end{abstract}

\maketitle

\smallskip


%

\date{\today}


\section{Introduction}

Harary introduced the notions of a \textit{sum graph} and a
\textit{difference graph} in \cite{harary90}. In the first case, one
has a graph that can be labelled with elements of $\NN$, the set of
positive integers, in such a way that two vertices are adjacent if and
only if the sum of the labels of those vertices is a label of some
other vertex, while in the second case the absolute value of the
diference replaces the sum.  In Section $4$ of \cite{harary90}
(p.~105), Harary remarks that the notion of a sum graph can be
naturally extended to other ``number systems'', which in its wider
generality are magmas (cf.~Definition 1, p.~1 in
\cite{bourbaki1998algebra}).

In Section 2, we define the notion of a sum graph over a magma, as
well as the notion of a relaxed sum graph, giving some examples.

In Section 3, we deal with sum graphs over magmas with operations on
sets, showing that $C_n$, the cycle of length $n$, is a sum graph, for
$n\geq 4$, exactly when $n$ is even; characterising when the complete
graph, $K_m$, is a sum graph for the symmetric difference; and showing
that $C_4$ is not a sum graph for the magmas over sets with the
operation given by the complement of intersection, or of the union.

In Sections 4 to 8 we investigate when are cycle graphs sum graphs
over finite abelian groups. In particular, we show that $C_4$ is a sum
graph over an abelian group if and only if the order of the group is a
multiple of $5$; find a necessary condition for $C_n$ to be a
sum graph over an abelian group, and use this to give several
examples. Finally, we show that $C_{4\ell}$ is always a sum group over
an appropriate abelian group.

An integral sum graph, first defined in \cite{harary94}, is nothing
more that a sum graph over the additive group of the
integers. Melnikov and Pyatkin \cite{MelnikovP02} introduced the
notion of the radius of an integral sum graph, and in Section 9 we
give a linear upper bound for the radius of matchings, improving
Harary's labelling for this family of graphs. We also give the exact
radius for the family of totally disconnected graphs.  Using the Z3
problem solver, we found the exact number of integral sum graphs, and
of relaxed integral sum graphs, for all graphs with size up to 8, as
well as the number of integral for the cubic graphs with 4 to 12
vertices, and also integer labellings for the 4D-cube, $Q_4$, giving a
negative answering to a question of Melnikov and Pyatikin. Moreover,
we give the radius of $Q_4$, and show that it has an infinite number
of primitive labellings.

In Section 10, we indicate that $K_{3,3}$ is not a mod sum graph,
while $Q_3$ and the Petersen graph are mod sum graphs for appropriate
moduli. We give an explicit upper bound for the minimal modulus for
which a connected graph admits a sum labelling.

In Section 11, we report that $K_{3,3}$ is a relaxed mod sum graph,
while the triangular prism is not, and also give the exact number of
relaxed integral sum graphs for all cubic graphs with 4 to 14
vertices.

Finally, in Section 12, we show that the direct product of two strong
integral sum graphs is also a strong integral sum graph.

\section{Sum Graphs over Magmas}

The relevance of the notion of a sum graph is that it constitutes a
concise description for the graphs that admit a suitable
labelling. For this purpose, a magma is the more general structure
that accomodates this representation. Recall that a magma is just a
set equipped with a binary operation.

\begin{Def}
  \label{sumgraph}
  Given a magma $(M,\oplus)$ and $V\subseteq M$, we will denote by
  $\mathcal{G}_{(M,\oplus)}(V)$, or simply by $\mathcal{G}_{M}(V)$
  when the operation on $M$ is clear from the context, the (simple)
  graph $G=(V,E)$ whose edges are given by:
  \begin{equation}
    \label{sumgraph_c}
    (v,w)\in E \iff v\oplus w\in V \;\;\vee\;\; w\oplus v\in V.
  \end{equation}
  
  A graph $G$ is a \emph{sum graph over $M$}, or an \emph{$M-$graph},
  when $G=\mathcal{G}_M(V)$ for some $V\subseteq M$. A graph is said
  to be a \emph{strong $M-$graph} when, for any $v\in V$, $v\oplus
  v\not\in V$.

  A \emph{labelling} of a graph $G=(V,E)$ on a magma $M$ is a map
  $\lambda:V\to W$ with $W\subseteq M$ such that the graphs $G$ and
  $\mathcal{G}_{M}(W)$ are isomorphic. The labelling is called a
  \emph{strong labelling} if $\mathcal{G}_{M}(W)$ is a strong
  $M-$graph.
\end{Def}

\noindent\textit{Examples:}
\begin{enumerate}
\item The following is an example of a $(\ZZ,+)-$graph with the
  vertices identified with their labels, as we will always do
  throughout this paper:
  \begin{center}
    \begin{tikzpicture}
      [node distance=15mm, scale=1, auto=left]%
      \node(n1) {-3};%
      \node(n2) [right of=n1] {1}; %
      \node(n3) [below of=n2] {-2};%
      \node(n4) [left of=n3] {-1};%
      \node(n5) [right of=n2] {-4};%
      \node(n6) [right of=n3] {3};%
      \foreach \from/\to in {n1/n2, n2/n3, n3/n4,
        n4/n1,n2/n5,n3/n6,n5/n6} \draw (\from) -- (\to);
    \end{tikzpicture}
  \end{center}
 
\item In~\cite{harary94}, Harary points out that, letting
  $(a_n)_{n\in\NN}$ be inductively defined by $a_0=1$, $a_1=2$,
  $a_n = a_{n-2}-a_{n-1}$, for $n\geq 2$, the path graph, $P_n$, for
  $n\geq 4$, is a $(\ZZ,+)-$graph with the labelling given by this
  sequence. For instance:
  
\medskip

\begin{center}
  \begin{tikzpicture}
    [node distance=15mm, scale=1, auto=left]%
    \node(n0) {$P_6$:};%
    \node(n1) [right of=n0] {1}; %
    \node(n2) [right of=n1] {2};%
    \node(n3) [right of=n2] {-1};%
    \node(n4) [right of=n3] {3};%
    \node(n5) [right of=n4] {-4};%
    \node(n6) [right of=n5] {7};%
    \foreach \from/\to in {n1/n2, n2/n3, n3/n4, n4/n5,n5/n6} \draw
    (\from) -- (\to);
  \end{tikzpicture}
\end{center}
Harary does not provide a proof that this labelling does not generate
spurious edges. A proof was provided by Sharary for a family of close
related labellings (see Lemma 4 in \cite{sharary96}). Observe that the
maximum absolute of these labellings grows exponentially with the size
of the graph, but, as a consequence of the proof of Theorem 1 in
\cite{MelnikovP02}, it can be shown there exists a labelling that is
linear (see first paragraph in Section~\ref{radius} below).
\end{enumerate}

\noindent\emph{Remarks:}

\begin{itemize}
\item Every graph $G=(V,E)$ is an $M-$graph over some magma $M$. For
  example, one can define $M=\{v,\bullet\}$, for some fixed $v\in V$
  and the operation $\oplus$ on $M$ by: $a\oplus b=v$ if $(a,b)\in E$,
  and $a\oplus b=\bullet$ if $(a,b)\not\in E$.
\item Obviously, when $M$ is a commutative magma condition
  \eqref{sumgraph_c} simplifies to:
  $$(a,b)\in E \iff a\oplus b\in V.$$
\item A sum graph is simply a $(\NN,+)-$graph; an integral sum graph
  is a $(\ZZ,+)-$graph. We will simply call them $\NN-$graphs and
  $\ZZ-$graphs, respectively.
\item A difference graph \cite{harary90} is nothing else than an
  $(\NN,\ominus)-$graph, where $x\ominus y = |x-y|$.
\item A \emph{mod sum graph}, a concept introduced and studied in
  \cite{Boland90}, is a $(\ZZ_m,+)-$graph, for some $m\in\NN$ with
  $m\geq 2$ (the case $m=1$ is rather trivial). We will just call them
  $\ZZ_m-$graphs.
\item In \cite{deborah92} it was shown that $(\NN_{\geq
    2},\times)-$graphs are exactly the same as $(\NN,+)-$sum
  graphs.
\item In \cite{harary91reals} it was shown that a $(\RR^+,+)-$graph
  is a $(\NN,+)-$graph.
\item It is clear that any $\NN-$graph with more than one vertex
  cannot be connected, since the vertex with the maximum value cannot
  be adjacent to any other vertex. This leads to the notion of the
  \textit{sum number} of a given graph $G$, as the smallest number of
  isolated nodes which when added to $G$ yields a sum graph.  In
  \cite{harary94}, Harary states that the sum number of $C_n$ is $2$
  for all $n\geq 3$, except for $n=4$, in which case it is $3$. A
  proof of this statement was given in \cite[Theorem 3,
  p. 7]{sharary96}. As an example, the graph $C_4 \cup 3 K_1$ can be
  given as a $\NN-$graph in the following way, for instance:

\medskip

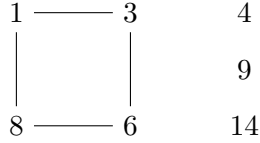
\begin{center}
  \begin{tikzpicture}
    [node distance=15mm, scale=1, auto=left]%
    \node(n1) {1};%
    \node(n2) [right of=n1] {3}; %
    \node(n3) [below of=n2] {6};%
    \node(n4) [left of=n3] {8};%
    \node(n5) [right of=n2] {4};%
    \node(n6) [below of=n5,yshift=7.5mm] {9};%
    \node(n7) [right of=n3] {14};%
    \foreach \from/\to in {n1/n2, n2/n3, n3/n4, n4/n1} \draw (\from)
    -- (\to);
  \end{tikzpicture}
\end{center}
\item According to Gallian \cite{gallian2022}, the notion of a strong
  sum graph was introduced in \cite{chang1993}. It was used in
  \cite{ellingham1993} to prove a conjecture of Harary, made in
  \cite{harary90}, that the sum number of every tree with at least two
  vertices is $1$.
\end{itemize}

The notion of a sum graph isomorphism is the natural one:

\begin{Def}\label{sumgraphiso}
  Given two sum graphs $\mathcal{G}_{(M,\oplus)}(V)$ and
  $\mathcal{G}_{(N,\odot)}(W)$, a bijective map $f:V\to W$ is said to
  be a \emph{sum graph isomorphism} if
  \begin{equation}\label{sumgraphiso_c}
    f(v_1\oplus v_2) = f(v_1) \odot f(v_2), \text{ for all } v_1,
    v_2\in V.
  \end{equation}
  If an isomorphim exists between the two sum graphs, we express that
  by writing
  $\mathcal{G}_{(M,\oplus)}(V)\simeq \mathcal{G}_{(N,\odot)}(W)$, and
  say that the two sum graphs are \emph{isomorphic}.
\end{Def}

The notion of a sum graph can be generalised by allowing distinct
vertices to have the same labelling, a concept that is obviously still
quite useful to give a concise description of a graph that allows such
a representation.

\begin{Def}\label{rxsumgraph}
  A \emph{relaxed sum graph over the magma $M$}, or
  $M-_{\rm rx}$\emph{graph} for short, is defined as in
  Definition~\ref{sumgraph} but by allowing $V$ to be a multiset with
  domain $M$.
\end{Def}

\noindent\textit{Examples:}
\begin{enumerate}
\item It is clear that every complete graph $K_n$ is a
  $\ZZ-_{\rm rx}$graph with the zero labelling, i.e.~all vertices are
  labelled with $0$. The fact that the complete graphs, a family
  easily described, are not integral sum graphs for $n\geq 4$
  (\cite{sharary96}) shows by itself the usefulness of the notion of
  relaxed sum graph.
\item Here is an example of a graph, a cubic graph, that is not a
  $\ZZ-$graph, but is a $\ZZ-_{\rm rx}$graph:
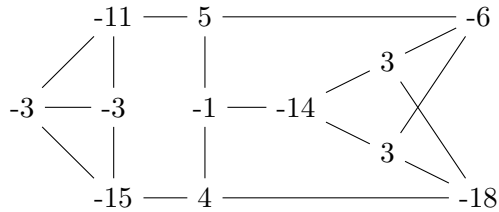
  \begin{center}
    \begin{tikzpicture}
      [node distance=15mm, scale=0.6, auto=left]%
      \node(n1) at (0,0) {-3};%
      \node(n2) at (2,-2) {-15};%
      \node(n3) at (2,2) {-11};%
      \node(n4) at (2,0) {-3};%
      \node(n5) at (4,2) {5};%
      \node(n6) at (4,0) {-1};%
      \node(n7) at (4,-2) {4};%
      \node(n8) at (6,0) {-14};%
      \node(n9) at (8,-1) {3};%
      \node(n10) at (8,1) {3};%
      \node(n11) at (10,2) {-6};%
      \node(n12) at (10,-2) {-18};%
      \foreach \from/\to in
      {n1/n2,n1/n3,n1/n4,n2/n4,n3/n4,n3/n5,n2/n7,n5/n6,n6/n7,n5/n11,
        n7/n12,n6/n8,n8/n9,n8/n10,n10/n12,n9/n11,n10/n11,n9/n12}
      \draw (\from) -- (\to);%
    \end{tikzpicture}
  \end{center}
  This example was found with the help of the Z3 Theorem Prover, as
  explained below, in Section~\ref{Z3section}.
\end{enumerate}

\section{Magmas on Sets}
\subsection{The Union and the Intersection}

Let $U(S)=(\pp(S),\cup)$ be the magma of the subsets of a set $S$ with
the union operation, and $I(S)=(\pp(S),\cap)$. The map
$\Psi:U(S)\to I(S)$ given by $\Psi(A)=\bar{A}$ is a magma isomorphism,
and thus a graph is a sum graph over $U(S)$ if and only if it is a sum
graph for $I(S)$.

For $U(S)$, the complete graphs $K_n$ are sum graphs for all
$n\leq |S|+1$, as it is enough to consider a chain of $n$ sets.

An example of a $C_4$ sum graph over $U(S)$, when $S$ has at least $4$
elements, is the following, where $A,B,C,D\subseteq S$ are four
disjoint subsets of $S$:

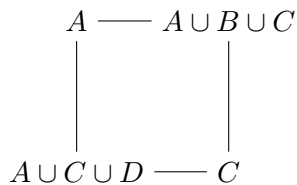
\medskip

\begin{center}
  \begin{tikzpicture}
    [node distance=2cm, scale=1, auto=left]%
    \node(n1) {$A$};%
    \node(n2) [right of=n1] {$A\cup B\cup C$}; %
    \node(n3) [below of=n2] {$C$};%
    \node(n4) [left of=n3] {$A\cup C\cup D$};%
    \foreach \from/\to in {n1/n2, n2/n3, n3/n4, n4/n1} \draw (\from)
    -- (\to);
  \end{tikzpicture}
\end{center}
Actually, one has:

\begin{prop}
  When $n\geq 5$ is an odd integer, the graph $C_n$ is not a sum graph
  over any magma of subsets of a set with the union operation. The
  graph $C_{2k}$ is a sum graph over that magma for a set with $2k$
  elements.
\end{prop}

\proof We start by noticing that, in such a sum graph that contains no
triangles, if $a$ and $b$ are two adjacent vertices, then $a\cup b$
cannot be distinct from both $a$ and $b$. Therefore, given two
adjacent vertices, one of them must contain the other.

Suppose now that $C_n$ is a sum graph, and let $v_1$ be a vertex that
is minimal for inclusion. Then $v_2$ contains $v_1$, and $v_3$ cannot
contain $v_2$ because that would create a triangle, so that $v_3$ must
be contained in $v_2$. The same reasoning applies to every consecutive
vertex, and therefore the Hasse diagram of the vertices must have the
shape:
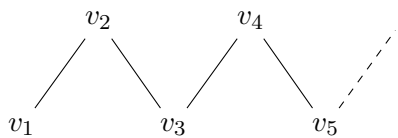
\begin{center}
  \begin{tikzpicture}
    [node distance=2cm, scale=1, auto=left]%
    \node(v1) {$v_1$};%
    \node(v2) [above right of=v1,xshift=-4mm] {$v_2$}; %
    \node(v3) [right of=v1] {$v_3$};%
    \node(v4) [right of=v2] {$v_4$};%
    \node(v5) [right of=v3] {$v_5$};%
    \node(v6) [right of=v4] {};%
    \path[-] (v1)edge(v2)%
    (v2)edge(v3)%
    (v3)edge(v4)%
    (v4)edge(v5);
    \draw[dashed] (v5) -- (v6);  
\end{tikzpicture}
\end{center}
There are two cases to consider: either $v_n$ is in the lower or in
the upper level of this diagram. In the first case, it cannot be
connected to $v_1$, as that would entail the edge $v_n v_2$, showing
that $C_{2n+1}$ is not a sum graph for $n\geq 2$. In the second case
one can give the following construction of $C_{2k}$: take pairwise
disjoint subsets $A_1, \ldots,A_k, B_1, \ldots, B_k$ (each can be a
singleton, for example) and set $v_{2i-1} = A_i$ and
$v_{2i}= A_i\cup A_{i+1}\cup B_i$, for $i=1,\ldots,k$, where we set
$A_{k+1} = A_1$.
\qed{}

\subsection{The Symmetric Difference}

Let $S$ be a finite set, and consider the group given by the symmetric
difference on the subsets of $S$. Note that the empty set is the
neutral element in this group, while every element is its own
inverse. A group in which every element, other than the identity, has
order 2 is called a \emph{Boolean group}, and it is easy to see that
every finite Boolean group is isomorphic to the group formed by the
subsets of some finite set, with the symmetric difference.

It is clear that, for all $n\in\NN$, the complete graph $K_{2^n}$ is a
sum graph for any of these groups. By removing the vertex
corresponding to the neutral element, one sees that the graph
$K_{2^n-1}$ is also a sum graph. We now show that there are no other
complete graphs that are sum graphs for Boolean groups.


\begin{prop}
  A graph $K_m$ is a sum graph over some magma of subsets of a set $S$
  with the symmetric difference if and only if $m = 2^k-\varepsilon$
  for $\varepsilon\in\{0,1\}$.
\end{prop}

\proof Suppose $K_m$ is a sum graph over the group $A$ of subsets of a
set $S$ with $n$ elements. Since in these groups each element is its
own inverse, and the operation restricted to the vertices of the graph
is closed, the subset formed by the vertices is either a subgroup of
$A$, if the empty set is one of the vertices, or, otherwise, becomes a
subgroup when adding the empty set. Thus, by Lagrange's theorem,
either $m$ or $m+1$ divides $2^n$. \qed
\subsection{The Complements of the Intersection and of the Union}

Consider the magmas $\overline{U}(S) = (\pp(S),\overline{\cap})$ and
$\overline{I}(S) = (\pp(S),\overline{\cup})$ on the subsets of a set
$S$ with the operations given by
$A\,\overline{\cap}\, B = \overline{A\cap B}$ and
$A\,\overline{\cup}\, B = \overline{A\cup B}$, respectively, where the
overline denotes the complement.

The map $\Psi:\overline{U}(S)\to \overline{I}(S)$ given by
$\Psi(A)=\bar{A}$ is a magma isomorphism, and thus a graph is a sum
graph over $\overline{U}(S)$ if and only if it is so for
$\overline{I}(S)$.

\begin{prop}
  The graph $C_4$ is not a sum graph over a magma $\overline{U}(S)$.
\end{prop}

\proof Suppose there is such a magma $\overline{U}(S)$, and let $A$
and $B$ be two adjacent vertices. Their composition
$\bar{A}\cup \bar{B}$ must be a third vertex, since
$\bar{A}\cup \bar{B} = A$ (for example) implies $A=S$ and
$B=\emptyset$, and then all edges incident to $\emptyset$ would be
present in the graph. Without loss of generality, we may assume that
the labelling of $C_4$ is as follows:
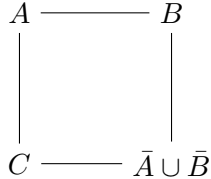
\begin{center}
  \begin{tikzpicture}
    [node distance=2cm, scale=1, auto=left]%
    \node(n1) {$A$};%
    \node(n2) [right of=n1] {$B$}; %
    \node(n3) [below of=n2] {$\bar{A}\cup\bar{B}$};%
    \node(n4) [left of=n3] {$C$};%
    \foreach \from/\to in {n1/n2, n2/n3, n3/n4, n4/n1} \draw (\from)
    -- (\to);
  \end{tikzpicture}
\end{center}
Now, $B\,\overline{\cap}\,(\bar{A}\cup\bar{B}) = \bar{B}\cup A$ is
either $C$ or $A$, because, as we just saw one cannot have
$X\,\overline{\cap}\, Y = X$, for any vertices $X$ and $Y$.

In the first case, $C=\bar{B}\cup A$, the edge $BC$ would be
present. In the other case, one would have $\bar{B}\subseteq A$, and
therefore $\bar{A}\subseteq B$, which entails that the diagonal edge
opposed to $BC$ would be present.  \qed

\section{The $C_4$ Problem on Abelian Groups}

As mentioned above, $C_4$ is not a sum graph, i.e.~an $\NN$-graph. It
is not even a $\ZZ-$graph, as stated by Harary in \cite{harary94} and
proved by Sharary in \cite[Theorem 3]{sharary96}. For a short and
elegant proof of this, see section 2 of \cite{MelnikovP02}.

Let us, then, consider the problem of determing the abelian groups $A$
such that $C_4$ is an $A-$graph.  Assume one has such a group $A$, and
let $V=\{a,b,c,d\}\subseteq A$. Since $0$, the neutral element of $A$,
cannot be in $V$, as it would be adjacent to every vertex, it is easy
to see that there are only two cases to consider:

\vspace{3mm}

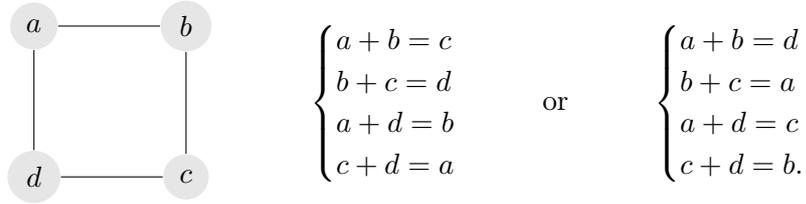
\begin{center}
  \begin{tabular}{cc}
    \lower 35pt\hbox{\begin{tikzpicture}
        [node distance=2cm, scale=1, auto=left, every
        node/.style={circle, fill=black!10}]%
        \node (n1) {$a$};%
        \node (n2) [right of=n1] {$b$}; %
        \node (n3) [below of=n2] {$c$};%
        \node (n4) [left of=n3] {$d$};%
        \foreach \from/\to in {n1/n2, n2/n3, n3/n4, n4/n1}
        \draw (\from) -- (\to);
      \end{tikzpicture}}\qquad{} %
    & $\begin{array}{ccc}
         \begin{cases}
           a+b=c\\
           b+c=d\\
           a+d=b\\
           c+d=a
         \end{cases}
         &\quad\text{or}\qquad
         &
           \begin{cases}
             a+b=d\\
             b+c=a\\
             a+d=c\\
             c+d=b.
           \end{cases}
       \end{array}
    $
  \end{tabular}
\end{center}

From these one deduces that one must have, in both cases, $5 a=0$
(with $a\neq 0$), and, respectively,
\begin{equation*}
  \begin{array}{ccc}
    \begin{cases}
      b=3a\\
      c=4a\\
      d=2a
    \end{cases}
    &\quad\text{or}\qquad
    &
      \begin{cases}
        b=2a\\
        c=4a\\
        d=3a.
      \end{cases}
  \end{array}
\end{equation*}

Note that both give rise to the same set of vertices and, in both
cases, the sums of the diagonally opposed vertices are $0$, which
ensures that resulting graph is indeed a $C_4$. This proves the
following.

\begin{prop}
  \label{C4groups}
  If $A$ is an abelian group of order $n$, then $C_4$ is a $A-$graph
  if and only if $5 \mid n$.
\end{prop}

\textit{Example}: Choosing $a=1$ in $\ZZ_5$ and $a=3$ in $\ZZ^*_{11}$,
one obtains the following isomorphic sum graphs over these groups:
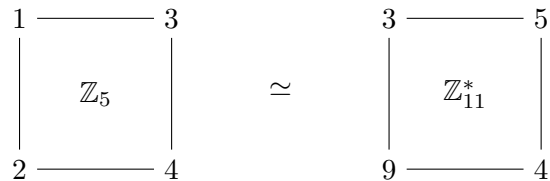
\begin{center}
  \begin{tikzpicture}
    [node distance=2cm, scale=1, auto=left]%
    \node(n1) {1};%
    \node(n2) [right of=n1] {3}; %
    \node(n3) [below of=n2] {4};%
    \node(n4) [left of=n3] {2};%
    \node at (1,-1) {$\ZZ_5$};%
    \foreach \from/\to in {n1/n2, n2/n3, n3/n4, n4/n1} \draw (\from)
    -- (\to);
    \end{tikzpicture}
    \hspace{8mm} \raise 35pt\hbox{$\simeq$} \hspace{8mm}
    \begin{tikzpicture}
    [node distance=2cm, scale=1, auto=left]%
    \node(n1) {3};%
    \node(n2) [right of=n1] {5}; %
    \node(n3) [below of=n2] {4};%
    \node(n4) [left of=n3] {9};%
    \node at (1,-1) {$\ZZ^*_{11}$};%
    \foreach \from/\to in {n1/n2, n2/n3, n3/n4, n4/n1}
    \draw (\from) -- (\to);
  \end{tikzpicture}
\end{center}

\section{The $C_n$ Problem}

Given $n\in\NN$, let us consider the problem of finding an abelian
group $A$ for which $C_n$ is a $A-$graph. We will restrict our search
to \emph{Fibonacci labellings} of cyclic graphs, i.e. labbelings in
which each label is the sum of the two previous labels, when going
through the vertices sequentially, clockwise or otherwise. Here is an
example of such a Fibonacci labelling of $C_7$ over the group
$\ZZ_{29}$, and a non-Fibonacci labelling of that same graph, but over
$\ZZ_{7}$, in which we used SageMath~\cite{sage} to confirm the
non-existence of spurious edges:
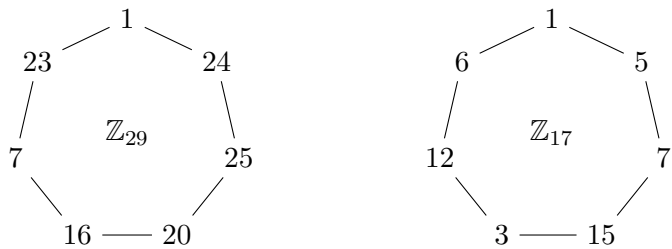
\begin{center}
  \begin{tabular}{c@{\hskip 2cm}c}
    \begin{tikzpicture}[roundnode/.style={circle, fill=black!10}]
      \node[draw=none,minimum size=3cm,regular polygon,regular polygon
      sides=7] (a) {};
      \node(n1) at (a.corner 1){$1$};%
      \node(n2) at (a.corner 7){$24$};%
      \node(n3) at (a.corner 6){$25$};%
      \node(n4) at (a.corner 5){$20$};%
      \node(n5) at (a.corner 4){$16$};%
      \node(n6) at (a.corner 3){$7$};%
      \node(n7) at (a.corner 2){$23$};%
      \node at (0,0) {$\ZZ_{29}$};%
      \foreach \from/\to in {n1/n2, n2/n3, n3/n4, n4/n5, n5/n6, n6/n7,
        n7/n1} \draw (\from) -- (\to);
    \end{tikzpicture}
    &
      \begin{tikzpicture}[roundnode/.style={circle, fill=black!10}]
        \node[draw=none,minimum size=3cm,regular polygon,regular
        polygon sides=7] (a) {};
        \node(n1) at (a.corner 1){$1$};%
        \node(n2) at (a.corner 2){$6$};%
        \node(n3) at (a.corner 3){$12$};%
        \node(n4) at (a.corner 4){$3$};%
        \node(n5) at (a.corner 5){$15$};%
        \node(n6) at (a.corner 6){$7$};%
        \node(n7) at (a.corner 7){$5$};%
        \node at (0,0) {$\ZZ_{17}$};%
        \foreach \from/\to in {n1/n2, n2/n3, n3/n4, n4/n5, n5/n6,
          n6/n7, n7/n1} \draw (\from) -- (\to);
      \end{tikzpicture}
  \end{tabular}
\end{center}

If one denotes the vertices of the $n$-gon by $(a_i)_{i\in\ZZ_n}$, and
decide that $a_{i+1} = a_i + a_{i-1}$, one gets:
\begin{eqnarray*}
  a_2 &=& a_0 + a_1\\
  a_3 &=& a_1 + a_2\\
      &\vdots&\\
  a_{n-1} &=& a_{n-3} + a_{n-2}\\
  a_0 &=& a_{n-2} + a_{n-1}\\
  a_1&=& a_{n-1} + a_0.
\end{eqnarray*}
From these, it is easy to see that $a_i = f_{i-1} a_0 + f_i a_1$ for
$i = 2,\ldots, n-1$, $a_0 = f_{n-1} a_0 + f_n a_1$, and
$a_1 = f_n a_0 + f_{n+1} a_1$, where $f_n$ is the usual $n$-th
Fibonacci number. These last two equalities yield
\begin{equation}
  \label{eq:1}
  f_n a_1 =  (1-f_{n-1}) a_0 \quad \text{ and } \quad (f_{n+1}-1) a_1
  = - f_n a_0,
\end{equation}
which are, hence, necessary conditions for a Fibonacci labelling of
$C_n$ to exist.

Let $d = (f_n, f_{n+1}-1)$ and $x, y\in\ZZ$ such that
$d = x f_n + y (f_{n+1}-1)$. Let $z = (1-f_{n-1}) x -f_n y$. From
equations~\eqref{eq:1}, one obtains
\begin{equation}
  \label{eq:2}
  d a_1 = z a_0.
\end{equation}
Note that, since $d$ divides $f_n$ and $f_{n+1}-1$, it also divides
their difference, $f_{n-1}-1$ , and thus $d$ divides $z$.

Letting $q\in\NN$ be such that $f_n = dq$, one has
\begin{equation}
  \label{eq:3}
  (zq+f_{n-1}-1) a_0= 1.
\end{equation}
If one now lets $q_1\in\NN$ be such that $f_{n+1}-1 = d q_1$, then,
using the right equation in \eqref{eq:1} together with \eqref{eq:2},
one gets
\begin{equation}
  \label{eq:4}
  (z q_1+f_n) a_0=1.
\end{equation}
Setting $e=z q+f_{n-1}-1$ and $z_1=z q_1+f_n$, the last two equations
can be replaced by
\begin{equation}
  \label{eq:5}
  (e,z_1) a_0=1.
\end{equation}

In this way, equations \eqref{eq:1} may be replaced by equations
\eqref{eq:2} and \eqref{eq:5}. Observe that $d$ divides $(e,z_1)$.  We
can, then, summarize what was here shown, as follows.
\begin{prop}
  For $n\in\NN$, let $d=(f_n,f_{n+1}-1)$. Find $x,y\in\ZZ$ such that
  $d=x f_n + y(f_{n+1}-1)$, and set $z=(1-f_{n-1})x-f_n y$.  Let
  $q\in\NN$ be such that $f_n=dq$, and let $q_1\in\NN$ be such that
  $f_{n+1}-1 = dq_1$. Set $e=zq+f_{n-1}-1$, and $z_1=z q_1+f_n$.

  Then, a necessary condition for $a_0, a_1\in A$ to generate a
  Fibonacci labelling of $C_n$, over the abelian group $A$, is that
  the order of $a_0$ divides $(e,z_1)$, and that
  $a_1 = \frac{z}d\, a_0$.
\end{prop}

One can find $(f_n, f_{n+1}-1)$ as follows. By induction it is easy to
see that
\begin{equation}
  \label{eq:6}
  (f_n, f_{n+1}-1) = (f_{n-k}+(-1)^k f_k, f_{n-k-1}-(-1)^k f_{k+1}).
\end{equation}

Let us first consider the case when $n$ is odd, $n=2j+1$. Choosing
$k=j$ in the previous equality, one gets
\begin{eqnarray}
  \label{eq:7}
  (f_n, f_{n+1}-1) &=& (f_{j+1}+(-1)^j f_j, f_j-(-1)^j f_{j+1})\nonumber\\
                   &=& (f_{j+1}+f_j,f_{j+1}-f_j)\nonumber\\
                   &=& (2f_j+f_{j-1},f_{j-1})\nonumber\\
                   &=& (2f_j,f_{j-1}) \nonumber\\
                   &=& \begin{cases}2, & \text{if } 2\mid f_n \\
                     1, & \text{if } 2\nmid f_n \end{cases} =
                     \begin{cases}2, & \text{if } 3\mid n \\
                       1, & \text{if } 3\nmid n.
                   \end{cases}
\end{eqnarray}

Now, if $n$ is even, one needs to consider whether or not $4\mid
n$. If $n=4j$, then choosing $k=2j$ in equation \eqref{eq:6}, one has
\begin{eqnarray}
  \label{eq:8}
  (f_n, f_{n+1}-1) &=& (f_{2j}+ f_{2j}, f_{2j-1}-f_{2j+1})\nonumber\\
                   &=& (2f_{2j},f_{2j}) = f_{2j}.
\end{eqnarray}
If $n=4j+2$, then choosing $k=2j+1$ in equation \eqref{eq:6}, one has
\begin{eqnarray}
  \label{eq:9}
  (f_n, f_{n+1}-1) &=& (f_{2j+1}-f_{2j+1}, f_{2j}+f_{2j+2})\nonumber\\
                   &=& f_{2j}+f_{2j+2} = 2f_{2j}+f_{2j+1}.
\end{eqnarray}
Summarising:
\begin{prop}
  For all $n\in\NN$,
  \begin{equation*}
    (f_n, f_{n+1}-1) =
    \begin{cases}2, & \text{if } n\equiv 3\pmod{6}, \\
      1, & \text{if }  n\equiv 1,5 \pmod{6},\\
      f_{\frac{n}2}, & \text{if } n \equiv 0 \pmod{4},\\
      f_{\frac{n}2}+2f_{\frac{n}2-1}, & \text{if } n \equiv 2 \pmod{4}.
    \end{cases}
  \end{equation*}
\end{prop}

\medskip

Using this, and with the help of SageMath to verify that one does not
get undesired edges, we got the following two examples:
\begin{itemize}
\item For $n=15$, one is in the first case of equation \eqref{eq:7},
  so that $d=2$. Also, $z=162$, $(e,z_1)=682$. Taking $A=\ZZ_{682}$,
  $a_0=1$ e $a_1 = 81$, one obtains the following example of a
  labelling for $C_{15}$:

  \noindent
  $\mathcal{G}_{\ZZ_{682}}([1, 81, 82, 163, 245, 408, 653, 379, 350,
  47, 397, 444, 159, 603, 80])$,

  \noindent
  the vertices being displayed here in cyclic order.
\item For $n=6$, one gets $d=4$, $z=-4$, $(e,z_1)=4$. Using the group
  $\ZZ_4\times\ZZ_4$ with $a_0=(1,2)$ and $a_1=(0,1)$, the conditions
  \eqref{eq:2} and \eqref{eq:5} are satisfied, and one indeed gets a
  Fibonacci labelling for $C_6$ over this group:

  
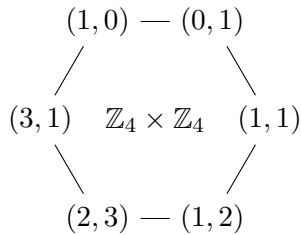
\begin{center}
    \begin{tikzpicture}[roundnode/.style={circle, fill=black!10}]
      \node[draw=none,minimum size=3cm,regular polygon,regular polygon
      sides=6] (a) {};
      \node(n1) at (a.corner 1){$(0,1)$};%
      \node(n2) at (a.corner 2){$(1,0)$};%
      \node(n3) at (a.corner 3){$(3,1)$};%
      \node(n4) at (a.corner 4){$(2,3)$};%
      \node(n5) at (a.corner 5){$(1,2)$};%
      \node(n6) at (a.corner 6){$(1,1)$}; \node at (0,0)
      {$\ZZ_4\times\ZZ_4$}; \foreach \from/\to in {n1/n2, n2/n3,
        n3/n4, n4/n5, n5/n6, n6/n1} \draw (\from) -- (\to);
    \end{tikzpicture}
  \end{center}
\end{itemize}
\subsection{A Serendipitious Conjecture}

For $n\in\NN$, let $e(n)$ and $z_1(n)$ be defined and above, and set
$\delta(n) = (e(n),z_1(n))$. Numerical computations make us suspect
that:
\begin{equation}
  \label{eq:10}
  \lim_{k\to\infty}\frac{\delta(n_{k+1})}{\delta(n_k)} =
  \begin{cases}
    2+\phi, & \text{ if } n_k = 4k+r, \text{ for } r\in\{0,2\}.\\[2mm]
    \ds 13+\frac{8}{\phi}, &\text{ if } n_k = 6k+r, \text{ for }
    r\in\{1,3,5\},
  \end{cases}
\end{equation}
where $\phi = \frac{1+\sqrt{5}}2$ is the golden ratio.
\section{More Examples}

In this section we gather some extra examples obtained with the help
of the results from the previous section, again using SageMath to
confirm that the edges are exactly the ones pretended.
\begin{itemize}
\item A Fibonacci labelling for $C_5$ over the group $\ZZ_{11}$:
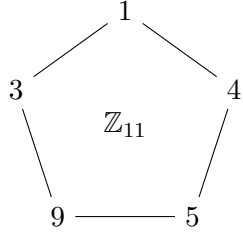
\begin{center}
  \begin{tikzpicture}[roundnode/.style={circle, fill=black!10}]
    \node[draw=none,minimum size=3cm,regular polygon,regular polygon
    sides=5] (a) {};
    \node(n1) at (a.corner 1){$1$};%
    \node(n2) at (a.corner 2){$3$};%
    \node(n3) at (a.corner 3){$9$};%
    \node(n4) at (a.corner 4){$5$};%
    \node(n5) at (a.corner 5){$4$};%
    \node at (0,0) {$\ZZ_{11}$};
    \foreach \from/\to in {n1/n2, n2/n3, n3/n4, n4/n5, n5/n1} \draw
    (\from) -- (\to);
  \end{tikzpicture}
\end{center}
\item A Fibonacci labelling for $C_6$ over the group
  $\ZZ_4\times \ZZ_4$, and a non-Fibonacci labelling for the same graph
  over $\ZZ_{13}$:
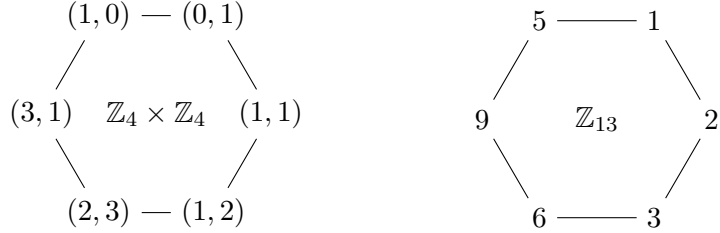
\begin{center}
  \begin{tabular}{c@{\hskip 2cm}c}
    \begin{tikzpicture}[roundnode/.style={circle, fill=black!10}]
      \node[draw=none,minimum size=3cm,regular polygon,regular polygon
      sides=6] (a) {};
      \node(n1) at (a.corner 1){$(0,1)$};%
      \node(n2) at (a.corner 2){$(1,0)$};%
      \node(n3) at (a.corner 3){$(3,1)$};%
      \node(n4) at (a.corner 4){$(2,3)$};%
      \node(n5) at (a.corner 5){$(1,2)$};%
      \node(n6) at (a.corner 6){$(1,1)$}; \node at (0,0)
      {$\ZZ_4\times\ZZ_4$}; \foreach \from/\to in {n1/n2, n2/n3,
        n3/n4, n4/n5, n5/n6, n6/n1} \draw (\from) -- (\to);
    \end{tikzpicture}
    &
      \begin{tikzpicture}[roundnode/.style={circle, fill=black!10}]
        \node[draw=none,minimum size=3cm,regular polygon,regular
        polygon sides=6] (a) {};
        \node(n1) at (a.corner 1){$1$};%
        \node(n2) at (a.corner 2){$5$};%
        \node(n3) at (a.corner 3){$9$};%
        \node(n4) at (a.corner 4){$6$};%
        \node(n5) at (a.corner 5){$3$};%
        \node(n6) at (a.corner 6){$2$}; \node at (0,0) {$\ZZ_{13}$};
        \foreach \from/\to in {n1/n2, n2/n3, n3/n4, n4/n5, n5/n6,
          n6/n1} \draw (\from) -- (\to);
      \end{tikzpicture}
  \end{tabular}
\end{center}
\item A Fibonacci labelling for $C_8$ over the group
  $\ZZ_3\times\ZZ_{15}$, and one non-Fibonacci labelling over
  $\ZZ_{29}$:
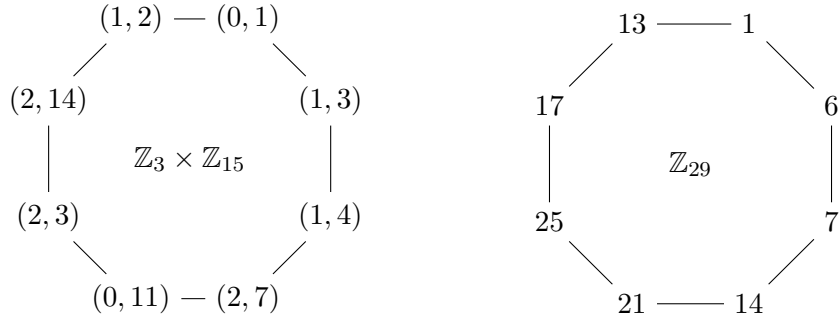
\begin{center}
  \begin{tabular}{c@{\hskip 2cm}c}
    \begin{tikzpicture}[roundnode/.style={circle, fill=black!10}]
    \node[draw=none,minimum size=4cm,regular polygon,regular polygon
    sides=8] (a) {};
    \node(n1) at (a.corner 1){$(0,1)$};%
    \node(n2) at (a.corner 8){$(1,3)$};%
    \node(n3) at (a.corner 7){$(1,4)$};%
    \node(n4) at (a.corner 6){$(2,7)$};%
    \node(n5) at (a.corner 5){$(0,11)$};%
    \node(n6) at (a.corner 4){$(2,3)$};%
    \node(n7) at (a.corner 3){$(2,14)$};%
    \node(n8) at (a.corner 2){$(1,2)$};%
    \node at (0,0) {$\ZZ_{3}\times\ZZ_{15}$};%
    \foreach \from/\to in {n1/n2, n2/n3, n3/n4, n4/n5, n5/n6, n6/n7,
      n7/n8, n8/n1} \draw (\from) -- (\to);
  \end{tikzpicture}
    &
      \begin{tikzpicture}[roundnode/.style={circle, fill=black!10}]
        \node[draw=none,minimum size=4cm,regular polygon,regular
        polygon sides=8] (a) {};
        \node(n1) at (a.corner 1){$1$};%
        \node(n2) at (a.corner 2){$13$};%
        \node(n3) at (a.corner 3){$17$};%
        \node(n4) at (a.corner 4){$25$};%
        \node(n5) at (a.corner 5){$21$};%
        \node(n6) at (a.corner 6){$14$};%
        \node(n7) at (a.corner 7){$7$};%
        \node(n8) at (a.corner 8){$6$};%
        \node at (0,0) {$\ZZ_{29}$};%
        \foreach \from/\to in {n1/n2, n2/n3, n3/n4, n4/n5, n5/n6,
          n6/n7, n7/n8, n8/n1} \draw (\from) -- (\to);
      \end{tikzpicture}
  \end{tabular}
\end{center}
\item Finally, a Fibonacci labelling for $C_{12}$ over the group
  $\ZZ_{40}\times\ZZ_{40}$:
  
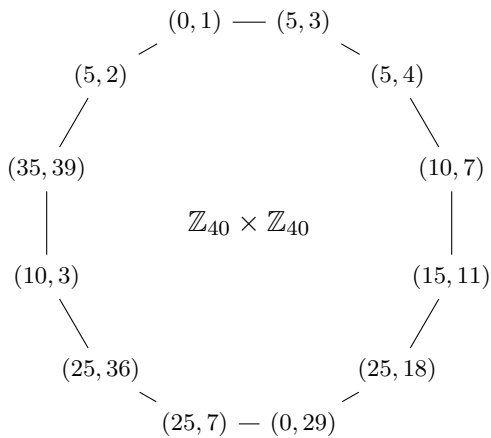
\begin{center}
    \begin{tikzpicture}[roundnode/.style={circle, fill=black!10}]
      \node[draw=none,minimum size=55mm,regular polygon,regular
      polygon sides=12] (a) {};
      \node(n1) at (a.corner 2){\footnotesize$(0,1)$};%
      \node(n2) at (a.corner 1){\footnotesize$(5,3)$};%
      \node(n3) at (a.corner 12){\footnotesize$(5,4)$};%
      \node(n4) at (a.corner 11){\footnotesize$(10,7)$};%
      \node(n5) at (a.corner 10){\footnotesize$(15,11)$};%
      \node(n6) at (a.corner 9){\footnotesize$(25,18)$};%
      \node(n7) at (a.corner 8){\footnotesize$(0,29)$};%
      \node(n8) at (a.corner 7){\footnotesize$(25,7)$};%
      \node(n9) at (a.corner 6){\footnotesize$(25,36)$};%
      \node(n10) at (a.corner 5){\footnotesize$(10,3)$};%
      \node(n11) at (a.corner 4){\footnotesize$(35,39)$};%
      \node(n12) at (a.corner 3){\footnotesize$(5,2)$};%
      \node at (0,0) {$\ZZ_{40}\times\ZZ_{40}$};%
      \foreach \from/\to in {n1/n2, n2/n3, n3/n4, n4/n5, n5/n6, n6/n7,
        n7/n8, n8/n9, n9/n10, n10/n11, n11/n12, n12/n1} \draw (\from)
      -- (\to);
    \end{tikzpicture}
  \end{center}
\end{itemize}
\section{A $C_9$ Example}

The attempts we made to find a number $m\in\NN$ for which $C_9$ is a
$\ZZ_m-$graph, using the method described above were unsucceful. Then,
to find such a number $m$, we used the following strategy.

If $V=\{ x_i: i\in\ZZ_9\}$ are the vertices of $C_9$, then one may seek
for a field in which the following linear system has a non-trivial
solution:
\begin{equation}
  \label{eq:11}
  x_i + x_{i+1} = x_{g(i)} \text{ where } g(i)\not\in \{i,i+1\} \text{
    and } g(i)\neq g(i+1),
\end{equation}
for all $i\in\ZZ_9$.

Using SageMath to randomly search for integral matrices $9\times 9$
corresponding to linear systems like this, computing their
determinants to find a finite field where these are zero, then
computing a non-trivial element of the respective kernel, and finally
checking that the sums of non-adjacent vertices are not in the chosen
solution, we were able to find the following example:
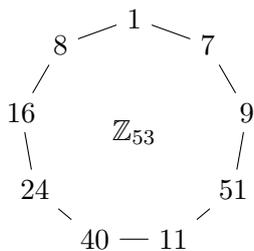
\begin{center}
  \begin{tikzpicture}[roundnode/.style={circle, fill=black!10}]
    \node[draw=none,minimum size=3cm,regular polygon,regular polygon
    sides=9] (a) {};
    \node(n1) at (a.corner 1){$1$};%
    \node(n2) at (a.corner 2){$8$};%
    \node(n3) at (a.corner 3){$16$};%
    \node(n4) at (a.corner 4){$24$};%
    \node(n5) at (a.corner 5){$40$};%
    \node(n6) at (a.corner 6){$11$};%
    \node(n7) at (a.corner 7){$51$};%
    \node(n8) at (a.corner 8){$9$};%
    \node(n9) at (a.corner 9){$7$};%
    \node at (0,0) {$\ZZ_{53}$};%
    \foreach \from/\to in {n1/n2, n2/n3, n3/n4, n4/n5, n5/n6, n6/n7,
      n7/n8, n8/n9, n9/n1} \draw (\from) -- (\to);
  \end{tikzpicture}
\end{center}
as well as two others over the field $\ZZ_{47}$:
\begin{eqnarray*}
  \mathcal{G}_{\ZZ_{47}}([1,12,36,23,13,30,43,26,22]),\\
  \mathcal{G}_{\ZZ_{47}}([1,34,14,26,35,38,44,41,40]).
\end{eqnarray*}

\section{$C_{4\ell}$ as a Sum Graph over an Abelian Group}

Set $n=4\ell$, $f=f_{2\ell}$ and $A=\ZZ_f\times\ZZ_f$. Consider the
graph $C$ whose vertices are $a_i = (f_i,f_{i-1})\in A$ with
$i=0,1,\ldots,n-1$, and the edges the ones induced by the sum graph
law. By~\eqref{eq:8}, and using the same notations as above, we get
$d= f_{2\ell} =f$. Since $d$ divides $z$, equations~\eqref{eq:2} and
\eqref{eq:5} are trivially satisfied in $A$. Therefore, it only
remains to show that $C$ does not have any edge besides the
$a_ia_{i+1}$.

Suppose, by contradiction, that there is, in $C$, an edge $a_ia_j$
with $0<i+1<j<n$, i.e.~that $a_i+a_j=a_k$ in $A$. This entails
\begin{equation*}
  \left\{
    \begin{array}{l}  
      f_i + f_j \equiv f_k \pmod{f}\\
      f_{i-1} + f_{j-1} \equiv f_{k-1}\pmod{f}.
    \end{array}
  \right.
\end{equation*}
Subtracting the second from the first, mantaining the first, and
swapping them, one gets
\begin{equation*}
  \left\{
    \begin{array}{l}  
      f_{i-1} + f_{j-1} \equiv f_{k-1}\pmod{f}\\
      f_{i-2} + f_{j-2} \equiv f_{k-2}\pmod{f}.
    \end{array}
  \right.
\end{equation*}
Repeating this, one eventually obtains $a_0 + a_{j-i} = a_{k-i}$,
which means that one can reduce to the case $i=0$.  Thus, one has
\begin{equation}
  \label{eq:13}
  \left\{
    \begin{array}{l}  
      f_j \equiv f_k \pmod{f}\\
      1 + f_{j-1} \equiv f_{k-1}\pmod{f}.
    \end{array}
  \right.
\end{equation}

Taking $j=j'+2\ell$ and $k=k'+2\ell$, with $1<j'<2\ell$ and
$-2\ell<k'<2\ell$, which can be done since
$f_i\equiv f_{i+4\ell} \pmod{f_{2\ell}}$ (this follows from
\eqref{eq:13} and $f_{4\ell+1}\equiv 1\pmod{f_{2\ell}}$), and using
the fact that (\cite{vorobiev}, (1.8), p.~9)
\begin{equation}
  \label{eq:vorobiev}
  f_{m+n} = f_{m-1} f_n + f_m f_{n+1},
\end{equation}
which holds for all $m,n\in\ZZ$, one obtains from the first
congruence in \eqref{eq:13} that
\begin{equation*}
  f_{j'}f_{2\ell+1} \equiv f_{j'+2\ell} \equiv f_{k'+2\ell} \equiv
  f_{k'}f_{2\ell+1}\pmod{f_{2\ell}}.
\end{equation*}
Since two consecutive Fibonacci numbers are coprime, one deduces that
\begin{equation}
  \label{eq:14}
  f_{j'} \equiv f_{k'}\pmod{f_{2\ell}}.
\end{equation}

If $0< k'<2\ell$, this entails $f_{j'}=f_{k'}$, and hence $j=k$, which
is absurd. Now, since $f_{-t}=(-1)^{t-1}f_t$, one has to deal only
with the case when $k'$ is even. In this case, one gets
\begin{equation}
  \label{eq:15}
  f_{j'} + f_{k'}\equiv 0 \pmod{f_{2\ell}},
\end{equation}
with $1<j'<2\ell$, and $0<k'<2\ell-1$. This congruence can only hold
for $j'=2\ell-1$ and $k'=2\ell-2$. But, then, the second congruence in
\eqref{eq:13} yields
\begin{equation*}
  1 + f_{4\ell-2}\equiv f_{4\ell-3} \pmod{f_{2\ell}},
\end{equation*}
which is equivalent to
\begin{equation*}
  f_{4\ell-4}\equiv -1 \pmod{f_{2\ell}}.
\end{equation*}

Using \eqref{eq:vorobiev}, one gets:
$$-1\equiv f_{4\ell-4} \equiv f_{2\ell-1} f_{2\ell-4} \equiv f_{2\ell-2}
f_{-1}f_{2\ell-4}\equiv f_{2\ell-2} f_{2\ell-4}\pmod{f_{2\ell}}.$$
Applying again \eqref{eq:vorobiev},
$$-1\equiv f_{2\ell-2} f_{2\ell-4}\equiv f_{2\ell-1}f_{-2}f_{2\ell-4}
\equiv -f_{2\ell-1}f_{2\ell-4}\equiv 1  \pmod{f_{2\ell}},$$
which gives the desired contradiction.

We have thus shown the following result.
\begin{thm}
  For every $\ell\in\NN$, the graph $C_{4\ell}$ is a sum graph over
  the group $\ZZ_f\times\ZZ_f$, where $f=f_{2\ell}$ is the
  corresponding Fibonacci number.
\end{thm}

\section{Integral Sum Graphs}
\subsection{The radius of a $\ZZ-$graph.}
\label{radius}

For $\ZZ-$graphs, Melnikov and Pyatkin \cite{MelnikovP02} defined the
\emph{radius of a labelling} as its maximum absolute value, i.e.~the
radius of the smallest interval centered on the origin that contains
the labelling set, and defined the \emph{radius} of a $\ZZ$-graph,
$G$, as the smallest of the radius of all its labellings, denoted by
$r(G)$.  In their paper it is shown (see \cite[Theorem
1]{MelnikovP02}) that the radius of $C_n$ grows at most linearly with
$n$. It is easy to see that a labelling for $P_n$ can be obtained from
their labelling for $C_{n+1}$, by excluding the label of maximum
value, and that one has $r(P_n)\leq \frac{17}2 n$.

In \cite{harary94}, Harary showed that all matchings, $mP_2$, are
$\ZZ-$graphs. By relying on the ideas presented in \cite{MelnikovP02},
one can improve Harary's result by showing the following.

\begin{prop}
  For all $m\geq 4$, $r(mP_2)\leq 3m-4$.
\end{prop}

\begin{proof}
  The idea is to consider two intervals of integers,
  $$a,a+1,\ldots,a+(m-2),$$  $$-b,-b-1\ldots,-b-(m-2),$$
  with $a>b\in\NN$ appropriately chosen. To start with, choose $a$ and
  $b$ such that $2a\geq a+(m-1)$, $2b\geq b+(m-1)$, so that there are
  no edges between the vertices with labellings in the same
  interval. One then adds a pair of vertices with labellings $a-b$ and
  $b+(m-2)$, whose sum is the maximum label. The requirement that
  there are no undesired edges leads to the conditions $b\geq m-1$ and
  $a\geq 2m-2$. Choosing $b=m-1$ and $a=2m-2$ yields the following
  labelling for $mP_2$:

  \begin{center}
    \begin{tikzpicture}
      [node distance=25mm, scale=1, auto=left]%
      \node(n1) {$(2m-2)$};%
      \node(n2) [right of=n1] {$(2m-1)$}; %
      \node(n3) [below of=n2,yshift=10mm] {$(-m)$};%
      \node(n4) [left of=n3] {($-m+1$)};%
      \node(n5) [right of=n2,xshift=-10mm] {$\cdots$};%
      \node(n6) [right of=n3,xshift=-10mm] {$\cdots$};%
      \node(n7) [right of=n5,xshift=-10mm] {$(3m-4)$}; %
      \node(n9) [right of=n7] {$(m-1)$}; %
      \node(n8) [right of=n6,xshift=-10mm] {$(-2m+3)$}; %
      \node(n10) [right of=n8] {$(2m-3)$,}; %
      \foreach \from/\to in {n2/n3,n4/n1,n7/n8,n9/n10}
      \draw (\from) -- (\to);
    \end{tikzpicture}
  \end{center}
  which yields the claim.
\end{proof}

\noindent\emph{Remark:} Rupert Li, in \cite[Theorem 6.2]{Li}, gives ,
for $m\geq 3,$ the labelling
$$-1 \; ;\; 1, 3, 5,  \ldots, 4m-7, 4m-5\; ;\; 4m-4$$
for $mP_2$, which he proves to have the minimal range (the difference
between the biggest and the smallest labels), called \emph{ispum},
among all integral labellings for $mP_2$. Note that our labelling has
a bigger range, namely $5m-7$, but a smaller radius, showing that the
ispum and the radius are distinct characteristics of a (sum)
graph. The radius measures the smallest absolute magnitude of the
numbers needed to label the graph.

\bigskip

Sharary has shown that, for $n>3$, $K_n$ is not a $\ZZ-$graph
\cite[Theorem 2]{sharary96}. We show here that their complements,
$\bar{K}_n$, are $\ZZ-$graphs, for all $n$, and determine their radii.

\begin{prop}
  For all $n\in\NN$, $r(\bar{K}_n) = n-1$.
\end{prop}

\begin{proof}
  It is easy to show that, for any $n\in\NN$, the numbers
  $$-(n-1), -(n-2), \ldots, -\left\lceil\frac{n}2\right\rceil,
  \left\lfloor\frac{n}2\right\rfloor,\left\lfloor\frac{n}2\right\rfloor+1,
  \ldots,(n-1)$$
  provide a labelling for $\bar{K}_n$, showing that
  $r(\bar{K}_n) \leq n-1$.

  To show that it cannot be smaller, suppose that there is a labelling
  of $\bar{K}_n$ with labels in $[-(n-2),(n-2)]$. By multiplying by
  $-1$ if necessary, one may assume that there are at least
  $\lceil\frac{n}2\rceil$ labels in $[1,(n-2)]$, since $0$ cannot be
  one of the labels for $n>1$. We show by induction that the numbers
  $1,2,\ldots,\lfloor\frac{n}2\rfloor-1$ cannot be labels, which
  yields a contradiction, since
  $\lceil\frac{n}2\rceil+\lfloor\frac{n}2\rfloor-1 = n-1 > n-2$.  The
  proof is split into two cases: $n$ odd and $n$ even.

  \medskip

  \noindent\emph{Odd case}.  If $1$ is a label, then at least one element in
  each of the following pairs is not a label:
  $(2,3),(4,5),(6,7), \ldots (n-3,n-2)$, which would imply that there
  would be at most $n-2-\frac{n-3}2 =\frac{n-1}2$ labels.

  Assume, now, that one had already shown that any positive value
  smaller than $k$ is not a label, for any $k\leq \frac{n-3}2$. If $k$
  is a label, then at most one element of each of the following
  disjoint pairs is a label:
  \begin{center}
    \begin{tabular}{l}
      $(k+1,2k+1), (k+2,2k+2),\ldots, (2k,3k)$,\\
      $(3k+1,4k+1), (3k+2,4k+2),\ldots, (4k,5k)$,\\
      \hfil$\vdots$\hfill\\
      $((2i-1) k+1,2ik+1),\ldots, ((2i-1)k+j,2ik+j)$,
    \end{tabular}
  \end{center}
  with $i,j\in\NN$, where either
  \begin{itemize}
  \item[(a)] $j\leq k$ and $2ik+j = n-2$, or
  \item[(b)] $j=k$ and $(2i+1)k< n-2 < (2i+2)k+1$.
  \end{itemize}
  In case (a), the labelling would not include
  $(k-1)+k(i-1)+j = (n-2)-ki-1$ numbers, leaving out only $ki+1$,
  which is not greater than $\frac{n-1}2$, as $2ki+1\leq n-2$. In case
  (b), one would only have at most $(n-2) -ki-(k-1) = (n-1)-k(i+1)$
  possible labels, but this number is also less or equal to
  $\frac{n-1}2$, since
  $n-2<(2i+2)k+1\iff n-2\leq 2k(i+1) \iff n-1\leq 2k(i+1)$, since $n$
  is odd.

  \medskip

  \noindent\emph{Even case}. In this instance, if $1$ is a label, then
  at least one element in each of the following pairs is not a label:
  $(2,3),(4,5),(6,7), \ldots (n-4,n-3)$, and thus one has at most
  $\frac{n}2 -1$ labels up to $n-3$, consequentely $n-2$ must be a
  label. One concludes that $n-4$ must also be a label, and then
  $n-6$, and so on, up to $2$. For $n\geq 8$, this implis that the
  existence of the edge $2$---$4$, and the cases $n=2,4,6$ are easy to
  discard.

  Assume, now, that one had already shown that any positive value
  smaller than $k$ is not a label, for any $k\leq \frac{n}2-1$. Again,
  if $k$ is a label, then one element of each of the following
  disjoint pairs is not a label:
  \begin{center}
    \begin{tabular}{l}
      $(k+1,2k+1), (k+2,2k+2),\ldots, (2k,3k)$,\\
      $(3k+1,4k+1), (3k+2,4k+2),\ldots, (4k,5k)$,\\
      \hfil$\vdots$\hfill\\
      $((2i-1) k+1,2ik+1),\ldots, ((2i-1)k+j,2ik+j)$,
    \end{tabular}
  \end{center}
  with $i,j\in\NN$, where either
  \begin{itemize}
  \item[(a)] $j\leq k$ and $2ik+j = n-2$, or
  \item[(b)] $j=k$ and $(2i+1)k< n-2 < (2i+2)k+1$.
  \end{itemize}
  In case (a), the labelling does not include
  $(k-1)+k(i-1)+j = (n-2)-ki-1$ numbers, leaving out only $ki+1$,
  which is not greater than $\frac{n}2-1$, as $2ki+1\leq n-2$ implies
  $2ki+1\leq n-3$, as $n$ is even.

  The case (b) is a bit more delicate.  Considering now the following
  pairs, starting at $n-2$ and going downwards:
  \begin{center}
    \begin{tabular}{l}
      $(n-2,n-2-k),\ldots, (n-2-(k-1),n-2-(2k-1))$,\\
      $(n-2-2k,n-2-3k),\ldots, (n-2-(3k-1),n-2-(4k-1))$,\\
      \hfil$\vdots$\hfill\\
      \footnotesize
      $(n-2-2(i-1)k,n-2-(2i-1)k),\ldots, (n-2-((2i-1)k-1),n-2-(2ik-1))$,
    \end{tabular}
  \end{center}
  one sees that the number of possible labels is, at most,
  $\ell=(n-2) -ki-(k-1) = (n-1)-k(i+1)$. Since $n$ is even,
  $n-2<(2i+2)k+1$ implies $n-2\leq (2i+2)k$, and thus
  $\ell\leq \frac{n}2$. It follows that one must have
  $\ell = \frac{n}2$ and $n-2 = (2i+2)k$; that every pair in the above
  array must contain one (and only one) label; and that all numbers
  from $k+1$ up to $2k = (n-2)-2ik$ must be labels.
  But then, at least $k,k+1,2k$ are labels, since $k\geq 2$. It
  follows that $2k+1, 3k, 3k+1$ are not labels. Repeating the argument
  for the pairs starting at $(k+2,2k+2)$, in which case the last pair
  involved would be $((2i+1)k,(2i+2)k)$, one sees that one of the
  numbers in the pair $(2k+1,3k+1)$ must be a label. This yields the
  desired contradiction, as long as $3k+1\leq n-2 = 2(i+1)k$, which is
  equivalent to $\frac12+\frac{1}{2k}\leq i$, or $i\geq 1$. Finally,
  just note that the case $i=0$ occurs if and only if one has
  $\frac{n}2-1 = (i+1)k = k$, or $n-2=2k$, and then all numbers from
  $k$ to $2k$ must be labels. But then none of the numbers
  $-1,-2,\ldots,-k$ can be labels, and then there are not enough
  labels. This completes the proof.
\end{proof}

\noindent\emph{Remark:} For $n$ even, one also has, for $\bar{K}_n$, the
labelling with radius $n-1$:
$$-(n-1), -(n-3), \ldots,-3, -1, 1, 3,\ldots,(n-3), (n-1).$$
\subsection{The Z3 Theorem Prover and Presburger Arithmetic}
\label{Z3section}

The fact that the Presburger arithmetic is decidable implies that it
is decidable to find out if a given graph is or not a sum graph,
either over $\mathbb{Z}$ or $\mathbb{N}$. A problem of this kind can
be computationally solved by means of a SMT problem solver. A
Satisfiability Modulo Theories (SMT) problem is a decision problem for
logical formulas with respect to combinations of background theories,
in this case integer arithmetic. We used
Z3\footnote{\url{https://github.com/Z3Prover/z3/wiki}}, a SMT solver
that, for small instancies, could directly obtain answers for problems
concerning the existence of $(\ZZ,+)$ or $(\NN,+)-$graphs.

For $(\ZZ_n,+)-$graphs, Z3 documentation states that it does not
support, at the moment, modular arithmetic. But because these
problems, by definition, only involve, positive and negative,
equalities between additive expressions, the transformation of these
formulas into boolean formulas using integer arithmetic was
straightforward, making possible to use Z3 for these problems too.

Using Z3 and the list of graphs in Brendan McKay webpage\footnote{
  \url{http://users.cecs.anu.edu.au/~bdm/data/graphs.html}}, we got
the following data, where $t$ is the total number graphs with $n$
vertices, and $isg$ is the respective number of integral sum graphs:
\begin{center}
  \begin{tabular}{rrrrll}
    $n$&$t$&isg& risg &  $isg/t$ & $risg/t$\\ \hline
    2&2&2&2&1&1\\
    3&4&4&4&1&1\\
    4&11&5&6&0.45&0.55\\
    5&34&14&18&0.41&0.53\\
    6&156&50&72&0.32&0.46\\
    7&1044&226&361&0.22&0.35\\
    8&12346&1460&3162&0.12&0.26\\
  \end{tabular}
\end{center}
It is a direct consequence of the fact that a sum labelling
constitutes a compressing code for a graph that admits such a
labelling, that the percentage of graphs that are sum graphs must go
to zero as the size of the graph increases, by a basic Kolmogorov
complexity result \cite{LIVitanyi2008}.


\subsection{The $n$-dimensional Cube}

In \cite{MelnikovP02}, Melnikov and Pyatkin showed that the 3D-cube
graph, $Q_3$, is not an integral sum graph, and then asked (Question
3, p.~245) whether it is true that the $n$-dimensional cube graph,
$Q_n$, is not an integral sum graph for all $n\geq 2$. It turns out
that the answer to this question is negative. Using the Z3 Theorem
Prover we have obtained the integral sum labellings of the 4D-cube
that we are going to describe using the order of the vertices depicted
in the following figure:

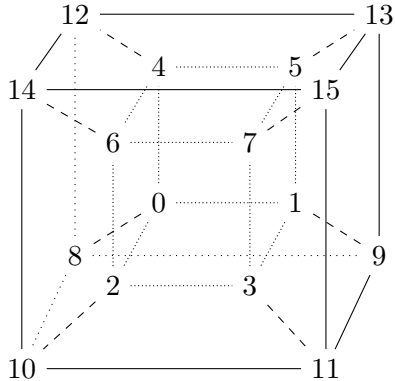
\vspace{4mm}

\begin{center}
  \begin{tikzpicture}
    [node distance=2cm, scale=1, auto=left]%
    \node(n0) at (-0.2,0.2) {0};%
    \node(n1) at (1.6,0.2) {1};%
    \node(n2) at (-0.8,-0.9) {2};%
    \node(n3) at (1,-0.9) {3};%
    \node(n4) at (-0.2,2) {4};%
    \node(n5) at (1.6,2) {5};%
    \node(n6) at (-0.8,1) {6};%
    \node(n7) at (1,1) {7};%
    \node(n8) at (-1.3,-0.5) {8};%
    \node(n9) at (2.7,-0.5) {9};%
    \node(n10) at (-2,-2) {10};%
    \node(n11) at (2,-2) {11};%
    \node(n12) at (-1.3,2.7) {12};%
    \node(n13) at (2.7,2.7) {13};%
    \node(n14) at (-2,1.7) {14};%
    \node(n15) at (2,1.7) {15};%
    \foreach \from/\to in
    {n9/n11,n10/n11,n10/n14,n11/n15,n9/n13,n12/n13,
      n12/n14,n13/n15,n15/n14} \draw (\from) -- (\to);%
    \foreach \from/\to in
    {n8/n9,n8/n10,n8/n12} \draw[dotted] (\from) -- (\to);%
    \foreach \from/\to in {n0/n1,n0/n2,n1/n3,n2/n3,n0/n4,n1/n5,n2/n6,
      n3/n7,n4/n6,n5/n7,n4/n5,n6/n7} \draw[densely dotted] (\from) --
    (\to);%
    \foreach \from/\to in
    {n0/n8,n1/n9,n2/n10,n3/n11,n4/n12,n5/n13,n6/n14, n7/n15}
    \draw[dashed] (\from) -- (\to);%
  \end{tikzpicture}
\end{center}
Note that here the tags are not labels in the previous sense, i.e.~sum
labellings, but are just meant to give a specific order to the
vertices --- if you convert the tags to binary, you will see the
reason behind the choice we made. Using this order, the first three
labellings found by Z3 were, with the respective radius:
$$[-17,38,6,-46,-21,-19,-25,8,-8,-32,-38,21,-11,-6,19,-40]_{r=46}$$
$$[29,-10,-32,18,-19,5,37,-8,8,-22,10,19,-3,32,-5,-14]_{r=37}$$
$$[8,18,-6,-29,-26,-9,-3,11,-11,-24,-18,26,2,6,9,-35]_{r=35}$$

These $3$ solutions are all in the nullspace, $K$, of the matrix
$M\in\mm_{32\times 16} (\ZZ)$ of the system $x_i + x_j - x_k = 0$,
where $v_iv_j$ is an edge of $Q_4$, with $i,j\in\{0,1,\ldots,15\}$
corresponding to their location in binary, and $v_k$ is the vertex
such that $v_i+v_j=v_k$. It turns out that $\dim K = 3$, with a
$\QQ-$basis given by:
$$
\arraycolsep=1pt
\begin{array}{ccccccccccccccccccccccccccccccccccc}
  u_1&=&(&0&,&0&,&-1&,&1&,&0&,&0&,&1&,&-1&,&1&,&-1&,&0&,&0&,&-1&,&1&,&0&,&0&)\\
  u_2&=&(&-3&,&2&,&0&,&0&,&1&,&-1&,&-1&,&-2&,&2&,&-2&,&-2&,&-1&,&-3&,&0&,&1&,&0&)\\ 
  u_3&=&(&-1&,&0&,&0&,&1&,&1&,&0&,&0&,&-1&,&1&,&0&,&0&,&-1&,&-1&,&0&,&0&,&1&),
\end{array}
$$
and the above solutions are, respectively:
$$
\arraycolsep=1pt
\begin{array}{lcrcrcrr}
  \text{sol}_1 &=& -6 u_1 &+&19 u_2 &-& 40 u_3 &\qquad (\text{radius} = 46),\\
  \text{sol}_2 &=& 32 u_1 &-&5 u_2 &-& 14 u_3 & (\text{radius} = 37),\\
  \text{sol}_3 &=& 6 u_1 &+&9 u_2 &-& 35 u_3 & (\text{radius} = 35).
\end{array}
$$

We then used Z3 to probe for the radius of $Q_4$, which turns out to
be $24$, with the following labelling:
$$[-5,17,19,5,24,-12,-2,14,22,2,-24,12,-20,-4,4,-16]_{r=24}$$

Given a $\ZZ-$graph $G$ with a labelling $\lambda$, the nullspace of
the matrix of the system given by the equations $x_i+x_j = x_k$ with
$\{i,j\}\in E(G)$ and $k\in V(G)$ such that
$\lambda(i)+\lambda(j)=\lambda(k)$, will be here called \emph{the
  kernel of the labelling} $\lambda$ of $G$, while \emph{a kernel} of
a $\ZZ-$graph $G$ is the kernel of some of its labellings.

We now show that any $\ZZ-$graph with a kernel of dimension at least
$2$ has infinitely many \emph{primitive} labellings, i.e.~labellings
that are not obtainable by multiplying another labelling by some
integer of modulus bigger than one, using the fact that a vector space
over a infinite field (in the present case, $\QQ$) cannot be the union
of a finite number of its lower dimensional subspaces (see Theorem 1.2
in \cite{rotmanLinearAlgebra}).

\begin{thm}
  If a $\ZZ-$graph with at least two vertices has a kernel with
  dimension at least $2$, then it has infinitely many primitive
  $\ZZ$-labellings.
\end{thm}

\begin{proof} Let $G$ be a $\ZZ-$graph with a kernel $K$ with
  dimension at least $2$. Each non-edge of $G$ yields an equation
  $x_i+x_j = x_k $ that corresponds to an hyperplane that necessarily
  intersects $K$ in a lower dimensional subspace, since $G$ has a
  labelling, i.e.~a solution of the system $S$ that does not belong to
  neither of these hyperplanes. Since there are finitely many such
  hyperplanes, and their union cannot contain the entire space $K$, it
  follows that $K$ has infinitely many $\QQ-$lines outside all of
  those hyperplanes, each of which has an integral primitive point.
\end{proof}

As an immediate consequence of this result, and the labellings above
mentioned for the 4D-cube, we get:

\begin{cor}
  The hypercube graph $Q_4$ has infinitely many primitive labellings.
\end{cor}

We tried to use Z3 to determine whether $Q_5$ it is, or not, a
$\ZZ-$graph, for more than two months of CPU time, but without sucess.

\subsection{Cubic Graphs}

Using the Z3 Theorem Prover we obtained the following results:
\begin{itemize}
\item There are no cubic graphs with 4 or 6 vertices that are integral
  sum graphs.
\item Of the 5 non-isomorphic cubic graphs with 8 vertices, only one
  is an integral sum graph, as stated in \cite{MelnikovP02}.
\item Of the 19 non-isomorphic cubic graphs with 10 vertices, only 6
  are integral sum graphs, the ones given in \cite{MelnikovP02}. In
  particular, the Petersen graph is not an integral sum graph.
\item Of the 85 non-isomorphic cubic graphs with 12 vertices, only 9
  are not integral sum graphs.
\end{itemize}

\noindent{\bf Question:} Are all cubic graphs with sufficiently many
vertices integral sum graphs?
\section{Mod Sum Graphs}

A graph that is a $\ZZ_n-$graph for some $n\in\NN$ is known as a
\emph{mod sum graph}. Using the Z3 Theorem Prover, we obtained the
following results:
\begin{itemize}
\item Neither $K_{3,3}$, nor the triangular prism:
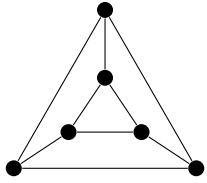
\begin{center}
  \begin{tikzpicture}
    [node distance=10mm, scale=0.6, auto=left, every
    node/.style={circle,inner sep=0pt, minimum size=6pt,fill=black!100}]]%
    \node(o1) at (2,3.5) {};%
    \node(o2) at (4,0) {};%
    \node(o3) at (0,0) {};%
    \foreach \from/\to in {o1/o2,o2/o3,o3/o1} \draw (\from) -- (\to);%
    \node(i1) at (2,2){};%
    \node(i2) at (2.8,0.8){};%
    \node(i3) at (1.2,0.8){};%
    \foreach \from/\to in {i1/i2,i2/i3,i3/i1} \draw (\from) -- (\to);%
    \foreach \from/\to in {o1/i1,o2/i2,o3/i3} \draw (\from) -- (\to);%
  \end{tikzpicture}
\end{center}
are mod sum graphs.
\item 
The 3D-cube is a mod sum graph over $\ZZ_{15}$:
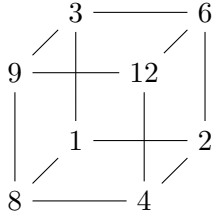
\begin{center}
  \begin{tikzpicture}
    [node distance=17mm, scale=0.8, auto=left]%
    \node(f1) {$9$};%
    \node(f2) [right of=f1] {$12$};%
    \node(f3) [below of=f2] {$4$};%
    \node(f4) [left of=f3] {$8$};%
    \foreach \from/\to in {f1/f2, f2/f3, f3/f4, f4/f1} \draw (\from)
    -- (\to);%
    \node(b1) at (1,1){$3$};%
    \node(b2) [right of=b1] {$6$};%
    \node(b3) [below of=b2] {$2$};%
    \node(b4) [left of=b3] {$1$};%
    \foreach \from/\to in {b1/b2, b2/b3, b3/b4, b4/b1} \draw (\from)
    -- (\to);%
    \foreach \from/\to in {f1/b1, f2/b2, f3/b3,f4/b4} \draw (\from) --
    (\to);%
  \end{tikzpicture}
\end{center}
\item 
The Petersen graph is a mod sum graph over $\ZZ_{28}$:
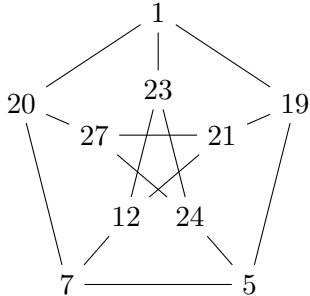
\begin{center}
  \begin{tikzpicture}
    [node distance=15mm, scale=0.6, auto=left]%
    \node(o1) at (4,6) {$1$};%
    \node(o2) at (1,4) {$20$};%
    \node(o3) at (7,4) {$19$};%
    \node(o4) at (2,0) {$7$};%
    \node(o5) at (6,0) {$5$};%
    \foreach \from/\to in {o1/o2,o2/o4,o4/o5,o5/o3,o3/o1} \draw
    (\from) -- (\to);%
    \node(i1) at (4,4.3){$23$};%
    \node(i2) at (2.6,3.3){$27$};%
    \node(i3) at (5.4,3.3){$21$};%
    \node(i4) at (3.3,1.5){$12$};%
    \node(i5) at (4.7,1.5){$24$};%
    \foreach \from/\to in {i1/i4,i4/i3,i3/i2,i2/i5,i5/i1} \draw (\from) --
    (\to);%
    \foreach \from/\to in {o1/i1,o2/i2,o3/i3,o4/i4,o5/i5} \draw (\from) --
    (\to);%
  \end{tikzpicture}
\end{center}
\end{itemize}

It turns out that, given a connected graph $G$, there is a number
$N\in\NN$ such that, if $G$ is not a $\ZZ_m-$graph for all
$m\leq N$, then $G$ is not a mod sum graph at all. In order to show
this, we start by noticing the followng trivial fact.

\begin{lem}
  Let $G$ be a $\ZZ_m-$graph with $n$ vertices and with labelling
  $v$. Set $d=\gcd(v_1,\ldots,v_n)$. If $d \mid m$, then $G$ is also a
  $\ZZ_{\frac{m}d}-$graph.
\end{lem}

\begin{proof}
  It immediately follows from the fact that
  $v_i+v_j\equiv v_k\pmod{m}$ is equivalent to
  $\frac{v_i}d + \frac{v_j}d \equiv \frac{v_k}d \pmod{\frac{m}d}$.
\end{proof}

\begin{thm}
  Let $G$ be a connected graph with $n\geq 3$ vertices, and set
  $N=2\cdot 3^{n-1}$. Then, if $G$ is not a $\ZZ_m-$graph for all
  $m\leq N$, then $G$ is not a mod sum graph at all.
\end{thm}

\begin{proof} Let $G$ be a connected graph with $n\geq 3$ vertices.
  Since all trees with $3$ or more vertices are mod sum graphs
  \cite{Boland90}, we may restrict our attention to graphs that are
  not trees, and hence have at least $n$ edges \cite[Theorem
  4.1]{hararyGT}.  Assume, then, that such a graph $G$ is a mod sum
  graph, and let $m\in\NN$ be the smallest number such that $G$ is a
  $\ZZ_m-$graph.  Let $M$ be the matrix of the homogeneous system
  $x_i+x_j - x_k = 0$ with $\{i,j,k\}\in\{1,2,\ldots,n\}$ such that
  $ij\in E_G$, where $v$ is the labelling function, and $k$ a vertex
  such that $v_i+v_j \equiv v_k \pmod{m}$. Let $A$ be an $n\times n$
  minor of $M$, and $\Delta\in\ZZ$ its determinant. By the usual proof
  of Cramer's rule, one has $\Delta v_i \equiv 0 \pmod{m}$, for all
  $i\in\{1,\ldots,n\}$. Let $d = \gcd(v_1,\ldots,v_n)$. If
  $(d,m)\neq 1$, by the previous lemma $G$ would be a
  $\ZZ_{\frac{m}{(d,m)}}-$graph, contradicting the minimality of
  $m$. Therefore $(d,m)=1$, and one concludes that $m\mid \Delta$.

  Finally, from the fact that the rows of $A$ consists of only three
  non-zero entries, two ones and one minus one, it is easy to conclude
  by induction, using the Laplace expansion to compute determinants,
  that $|\Delta| \leq 2\cdot 3^{n-1}$.
\end{proof}
\section{Relaxed Sum Graphs}
\subsection{Relaxed Integral Sum Graphs}

Again, using the Z3 Theorem Prover we have got:
\begin{itemize}
\item There are no cubic graphs with 4 or 6 vertices that are integral
  relaxed sum graphs.
\item Of the 5 non-isomorphic cubic graphs with 8 vertices, only one
  is an integral relaxed sum graph.
\item Of the 19 non-isomorphic cubic graphs with 10 vertices, only 6
  are non integral relaxed sum graphs. In particular the Petersen
  graph is (again) not an integral relaxed sum graph.
\item Of the 85 non-isomorphic cubic graphs with 12 vertices, only 2
  are not integral relaxed sum graphs.
\item All the 509 non-isomorphic cubic graphs with 14 vertices are
  integral relaxed sum graphs.
\end{itemize}

\noindent{\bf Question:} Are all cubic graphs with more that 12
vertices integral relaxed sum graphs?
\subsection{Relaxed Mod Sum Graphs}

The graph $K_{3,3}$ is a $\ZZ_9-_{\rm rx}$graph:
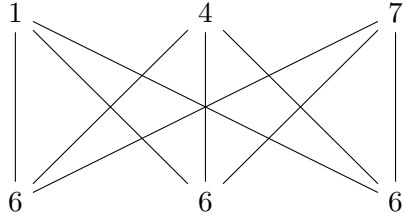
\begin{center}
  \begin{tikzpicture}
    [node distance=25mm, scale=1, auto=left]%
    \node(a1) {$1$};%
    \node(a2) [right of=a1] {$4$};%
    \node(a3) [right of=a2] {$7$};%
    \node(b1) [below of=a1] {$6$};%
    \node(b2) [below of=a2] {$6$};%
    \node(b3) [below of=a3] {$6$};%
    \foreach \from/\to in
    {a1/b1,a1/b2,a1/b3,a2/b1,a2/b2,a2/b3,a3/b1,a3/b2,a3/b3} \draw
    (\from) -- (\to);%
  \end{tikzpicture}
\end{center}

Using the Z3 Theorem Prover, we found out that the triangular prism is
not a relaxed mod sum graph.
 
\section{The Direct Product of Graphs}

The direct product $G\times H$ of two graphs $G$ and $H$, also known
as the tensor or Kronecker product, is the graph whose vertex set is
the cartesian product $G\times H$ and two vertices $(g_1,h_1)$ and
$(g_2,h_2)$ are adjacent if and only if $g_1$ is adjacent to $g_2$ in
$G$, and $h_1$ is adjacent to $h_2$ in $H$. We show here that if $G$
and $H$ are strong $\ZZ-$graphs, then so is $G\times H$.

We start with the following lemma, that is inspired in the proof of
the main result of \cite{deborah92}.

\begin{lem}
  If $G$ is a $\ZZ^k-$graph, then it is also a $\ZZ-$graph. The same
  holds for strong sum graphs.
\end{lem}

\begin{proof}
  Let $M$ be bigger than twice the maximum of the absolute values of
  all coordinates of the labels of the vertices of $G$. Then just use
  the fact that the map $\ZZ^k\to\ZZ$ given by
  $(x_1,\ldots, x_k) \mapsto \sum\limits_{i=1}^k x_i M^{i-1}$ is a
  group homomorphism, and it is injective when restricted to the
  labels of $G$, which follows from the fact that $x\equiv y\pmod{M}$
  implies that $x=y$ whenever $|x| < \frac{M}2$ and $|y| < \frac{M}2$.
\end{proof}

We can now show the claimed result.

\begin{thm}
  If $(G_i)_{i=1,\ldots,k}$ is a finite family of strong $\ZZ-$graphs,
  then so is their direct product $\prod\limits_{i=1}^k G_i$.
\end{thm}

\begin{proof}
  It is clear form the definition of the product of graphs that
  $\prod\limits_{i=1}^k G_i$ is a strong $\ZZ^k-$graph when every
  $G_i$ is a strong $\ZZ-$graph. The result then follows at once from
  the previous lemma.
\end{proof}

Observe that, if $G_1=(V_1,E_1)$ and $G_2=(V_2,E_2)$ are $\ZZ-$graphs,
but $G_1$ is not a strong $\ZZ-$graph, then the argument in the
previous proof fails, because if $u\in V_1$ is such that $2u\in V_1$,
and $v,w\in V_2$ are distinct, then $(u,v)+(u,w)\in V_1\times V_2$,
but $(u,v)$ is not adjacent to $(u,w)$ in $G_1\times G_2$. We point
out that Weischel \cite{weichsel62} has shown that $G\times H$ is
connected if and only if either $G$ or $H$ has an odd cycle, which
gives a way of constructing connected $\ZZ-$graphs.

\subsection*{Acknowledgements}

We thank Henning Fernau for introducing us to the subject of sum
graphs with an interesting question that we could not solve but that
motivated the work here presented.
\bigskip
\bibliographystyle{alpha}
\bibliography{SumGraphs}

\bigskip

\subsection*{Statements and Declarations}

This research was partially supported by CMUP (Centro de Matem\'atica
da Universidade do Porto), the Center for Mathematics of the
University of Porto, which is financed by national funds through FCT
under the project with reference UID/MAT/00144/2020.

The authors have no relevant financial or non-financial interests to
disclose.

\bigskip
\end{document}